\documentclass[11pt,a4paper,reqno]{amsart}
\usepackage[latin1]{inputenc}
\usepackage{amsmath}
\usepackage{amsfonts}
\usepackage{amssymb}
\usepackage{amsthm}
\usepackage{anysize}
\usepackage{enumitem}
\usepackage{amscd}
\usepackage{hyperref}
\usepackage[square, comma, numbers, sort&compress]{natbib}
\usepackage{indentfirst}
\marginsize{3.0cm}{3.0cm}{3.0cm}{3.0cm}
\setenumerate{label={\normalfont(\arabic*)}}

\newcommand{\form}[1]{{\langle #1 \rangle }}
\newcommand{\pfister}[1]{{\langle \! \langle #1 \rangle \! \rangle}}
\newcommand{\mydim}[1]{{\mathrm{dim}\!\; #1}}

\newcommand{\witti}[2]{{\mathfrak{i}_{#1}(#2)}}

\newcommand{\anispart}[1]{{#1_{\mathrm{an}}}}

\newcommand{\normform}[1]{{#1_{\mathrm{nor}}}}

\newtheorem{theorem}{Theorem}[section]

\newtheorem{lemma}[theorem]{Lemma}

\newtheorem{proposition}[theorem]{Proposition}
\newtheorem{corollary}[theorem]{Corollary}
\newtheorem{conjecture}[theorem]{Conjecture}

\theoremstyle{definition}
\newtheorem{definition}[theorem]{Definition}
\newtheorem{example}[theorem]{Example}

\newtheorem{problem}[theorem]{Problem}

\theoremstyle{remark}
\newtheorem{remark}[theorem]{Remark}
\newtheorem{remarks}[theorem]{Remarks}

\numberwithin{equation}{section}
\begin{document}

\title{Quasilinear quadratic forms and function fields of quadrics}
\author{Stephen Scully}
\address{Department of Mathematical and Statistical Sciences, University of Alberta, Edmonton AB T6G 2G1, Canada}
\email{stephenjscully@gmail.com}


\maketitle

\begin{abstract} Let $p$ and $q$ be anisotropic quadratic forms of dimension $\geq 2$ over a field $F$. In a recent article, we formulated a conjecture describing the general constraints which the dimensions of $p$ and $q$ impose on the isotropy index of $q$ after scalar extension to the function field of $p$. This can be viewed as a generalization of Hoffmann's Separation Theorem which simultaneously incorporates and refines some well-known classical results on the Witt kernels of function fields of quadrics. Using algebro-geometric methods, it was shown that large parts of this conjecture hold in the case where the characteristic of $F$ is not 2. In the present article, we prove similar (in fact, slightly stronger) results in the case where $F$ has characteristic $2$ and $q$ is a so-called \emph{quasilinear} form. In contrast to the situation where $\mathrm{char}(F) \neq 2$, the methods used to treat this case are purely algebraic. \end{abstract}

\section{Introduction} \label{secintro} Let $F$ be a field, let $p$ and $q$ be a pair of anisotropic quadratic forms of dimension $\geq 2$ over $F$, and let $F(p)$ be the function field of the integral projective $F$-quadric defined by the vanishing of $p$. The problem of understanding the isotropy behaviour of $q$ after scalar extension to the field $F(p)$ is one which lies at the heart of many of the central problems in the algebraic theory of quadratic forms. Let $\witti{0}{q_{F(p)}}$ denote the \emph{isotropy index} (i.e., the maximal dimension of a totally isotropic subspace) of $q$ extended to $F(p)$. In \cite{Scully3}, we formulated following conjecture which aims to describe the general constraints which the dimensions of $p$ and $q$ impose on the integer $\witti{0}{q_{F(p)}}$:
\begin{conjecture}\label{mainconj} Let $p$ and $q$ be anisotropic quadratic forms of dimension $\geq 2$ over a field $F$, and let $s$ be the unique non-negative integer such that $2^s < \mathrm{dim}(p) \leq 2^{s+1}$. Set $k = \mathrm{dim}(q) - 2 \witti{0}{q_{F(p)}}$. Then $k \geq 0$ and
$$ \mathrm{dim}(q)  = a2^{s+1} + \epsilon $$
for some non-negative integer $a$ and integer $-k \leq \epsilon \leq k$.  \end{conjecture}

\begin{remarks} \label{introrems}\begin{enumerate}[leftmargin=*] \item Note that we do not impose any assumption on the characteristic of $F$, and we permit $p$ and $q$ to be degenerate in the characteristic $2$ case.
\item The assertion is trivially true if $k \geq 2^s -1$, so we are interested in the case where $k \leq 2^s - 2$.
\item If $q$ is non-degenerate, then $\witti{0}{q_{F(p)}}$ is equal to the \emph{Witt index} of $q_{F(p)}$, i.e.,  half the dimension of its hyperbolic part. For degenerate forms, however, one has to distinguish between the isotropy index and Witt index. The reader is warned that our notation differs from \cite{EKM}, where $\mathfrak{i}_0$ is used to denote the Witt index.
\item If $q$ is non-degenerate, then the integer $k$ is just the dimension of the anisotropic part of $q_{F(p)}$. If $q$ is degenerate (in which case $\mathrm{char}(F) = 2$), this need not be true. Nevertheless, it is easy to check that $k \geq 0$ in all cases (see Lemma \ref{lempositived} below). \end{enumerate} \end{remarks}

It is not hard to see that Conjecture \ref{mainconj} is optimal, to the extent that there can be no further gaps in the possible values of $\mathrm{dim}(q)$ determined by $\witti{0}{q_{F(p)}}$ and $\mathrm{dim}(p)$ alone. Simple examples are given in \cite[Ex. 1.5]{Scully3} and Example \ref{exoptimal} below. From one point of view, the rough content of the conjecture is the following: The more isotropic $q$ becomes over the field $F(p)$, the closer $\mathrm{dim}(q)$ should be to being divisible by $2^{s+1}$. This does not capture the whole story, however; for example, when $\mathrm{dim}(q) \leq 2^s$, the conjecture asserts that $q$ remains anisotropic over $F(p)$, which is precisely the statement of the fundamental Separation Theorem originally discovered by Hoffmann in \cite{Hoffmann1}.

When $\mathrm{char}(F) \neq 2$, it was shown in \cite{Scully3} that Conjecture \ref{mainconj} holds if any of the following conditions are satisfied:
\begin{enumerate} \item $k < 2^{s-1}$;
\item $2^{s+1} - 2 \leq \mydim{p} \leq 2^{s+1}$;
\item $p$ is a Pfister neighbour;
\item $\mydim{p} \leq 16$;
\item $\mydim{q} \leq 2^{s+2} + 2^{s-1}$. \end{enumerate}
Perhaps most interesting here is the case where $k=0$, which tells us that if $q$ represents an element in the kernel of the natural restriction homomorphism $W(F) \rightarrow W(F(p))$ on Witt groups, then $\mathrm{dim}(q)$ is divisible by $2^{s+1}$. This new observation can be refined rather further (\cite{Scully1}), and perhaps raises some questions concerning the structure of the former kernel (about which very little is known in general).

The results of \cite{Scully3} are proved using methods from the theory of algebraic cycles, with the decisive tool being the action of Steenrod operations on the mod-2 Chow rings of certain smooth projective varieties. While the same ideas should, in principle, permit to produce analogous results in the case where $\mathrm{char}(F) = 2$ and $p$ and $q$ are non-degenerate, the absence of the mod-2 Steenrod operations in this setting renders this approach impractical at the present time. In fact, we currently have no practical approach to Conjecture \ref{mainconj} or other problems of its ilk for non-degenerate forms in characteristic $2$.\footnote{All recent advances in the characteristic $\neq 2$ which rely on the use of Steeenrod operations or algebraic cobordism theory remain open for non-degenerate forms in characteristic $2$.} The purpose of the present article is therefore to deal with a special class of degenerate forms in characteristic $2$ known as \emph{quasilinear} quadratic forms. Recall here that a \emph{quasilinear} quadratic form over a field of characteristic $2$ is one which is diagonalizable, i.e., isometric to a form of the shape $a_1x_1^2 + a_2 x_2^2 + \cdots + a_nx_n^2$, the $a_i$ being scalars in the field of definition. In characteristic $2$, the projective quadric defined by the vanishing of such a form is \emph{totally singular}, in the sense that it has no smooth points at all. Nevertheless, it was shown by Hoffmann and Laghribi (see, e.g., \cite{HoffmannLaghribi1}, \cite{HoffmannLaghribi2}, \cite{Hoffmann2}) that many aspects of the classical algebraic theory of quadratic forms can be extended to the study of quasilinear quadratic forms in characteristic $2$. This was elaborated upon in work of Totaro (\cite{Totaro1}) and the author (\cite{Scully0,Scully1,Scully2}), where quasilinear analogues of various non-trivial results of Karpenko, Merkurjev, Vishik and others on discrete invariants of quadratic forms in characteristic $\neq 2$ were studied, and proved in many cases. The methods employed here are of a rather different and more direct nature, and it is unclear to what extent they can be adapted to treat other cases. In fact, the quasilinear case is the only case in characteristic $2$ where some of these results are known.\footnote{For example, the analogue of Karpenko's theorem on the values of the first Witt index -- see \cite{Scully1}.} 

In the present work, we continue this theme by proving the following result towards the characteristic $2$ case of Conjecture \ref{mainconj}:

\begin{theorem} Assume, in the situation of Conjecture \ref{mainconj}, that $\mathrm{char}(F) = 2$ and $q$ is quasilinear. Then, the statement of the conjecture holds in the following cases:
\begin{enumerate} \item $k \leq 2^{s-1}$.
\item $2^{s+1} - 2^{s-2}-3 \leq \mathrm{dim}(p) \leq 2^{s+1}$.
\item $p$ is a quasi-Pfister neighbour.
\item $\mathrm{dim}(p) \leq 8$.
\item $2^N \leq \mathrm{dim}(q) \leq 2^N + 2^{s+1}$ for some non-negative integer $N$. 
\item $\mathrm{dim}(q) \leq 2^{s+2} + 2^{s+1}$.  \end{enumerate}
\end{theorem}

The reader will note that the results proved here are very similar (in fact, slightly stronger) than those proved in characteristic $\neq 2$ using completely different methods (\cite{Scully3}). The proofs for the first three cases are given in \S \ref{secmainresults} below, and for the last three in \S \ref{secbeyond}. In all cases, the basic tool is \cite[Thm. 6.6]{Scully2}, which seems to have many interesting applications to the study of quasilinear quadratic forms (see \cite[\S 6]{Scully2}). 

Before proceeding, we make the following important remark:

\begin{remark} Suppose that $\mathrm{char}(F) = 2$. If, in the situation of Conjecture \ref{mainconj}, $q$ is quasilinear, then we can assume that $p$ is quasilinear as well. Indeed, if $p$ is not quasilinear, then the quadric $\lbrace p=0 \rbrace$ is generically smooth, and so the extendion $F(p)/F$ is separably generated. By \cite[Prop. 5.3]{Hoffmann2}, it follows that $q$ remains anisotropic over $F(p)$, and so the statement of the conjecture holds trivially. When we speak of the \emph{quasilinear case of Conjecture \ref{mainconj}}, we shall therefore mean the case where \emph{both} $p$ and $q$ are quasilinear.\end{remark}

\noindent {\bf Conventions.} All quadratic forms considered in this paper are finite-dimensional. Non-degeneracy and regularity of quadratic forms are defined as in \cite[pp. 42-43]{EKM}.

\section{Recollections on quasilinear quadratic forms}

For the rest of the paper, we fix an arbitrary field $F$ of characteristic $2$. Let $\varphi$ be a quasilinear quadratic form over $F$. The $F$-vector space on which $\varphi$ is defined will be denoted $V_\varphi$. Given a field extension $L/F$, we write $\varphi_L$ for the unique quasilinear quadratic form on $V_\varphi \otimes_F L$ which restricts to $\varphi$ on $V_{\varphi}$. If $\psi$ is another quasilinear quadratic form over $F$, then we say that $\psi$ is a \emph{subform} of (resp. is \emph{isomorphic} to) $\varphi$ if there exists an injective (resp. bijective) $F$-linear map $f \colon V_\psi \rightarrow V_\varphi$ such that $\varphi\big(f(v)\big) = \psi(v)$ for all $v \in V_\psi$; in this case, we write $\psi \subset \varphi$ (resp. $\psi \simeq \varphi$). If $\psi \simeq a\psi$ for some $a \in F^*$, then we say that $\psi$ and $\varphi$ are \emph{similar}. The \emph{orthogonal sum} $\psi \perp \varphi$ and \emph{tensor product} $\psi \otimes \varphi$ are defined in the obvious way. We say that $\varphi$ is \emph{divisible} by $\psi$ if it is isomorphic to the tensor product of $\psi$ and another quasilinear quadratic form over $F$. If $a_1,\hdots,a_n \in F$, then we write $\form{a_1,\hdots,a_n}$ for the quasilinear quadratic form $a_1x_1^2 + \cdots + a_nx_n^2$ on the $F$-vector space $F^{\oplus n}$. Every quasilinear quadratic form over $F$ is isomorphic to one of this type. 

From now on, the terms ``quasilinear quadratic form'' and ``form'' will be used interchangeably. For the reader's convenience, we now quickly review some basic concepts and results which will be needed in the main part of the paper. Detailed introductions to the theory of quasilinear quadratic forms may be found in \cite{Hoffmann2} and \cite{Scully1}, and the unfamiliar reader is referred to these articles for further information.

\subsection{Isotropic decomposition} Let $\varphi$ be as above, and let $W \subset V_\varphi$ be the set of all isotropic vectors in $V_\varphi$, i.e., the set of all vectors $v \in V_\varphi$ such that $\varphi(v) = 0$. Since $\varphi$ is quasilinear, $W$ is an $F$-linear subspace of $V_\varphi$. Its dimension is the precisely the isotropy index  $\witti{0}{\varphi}$ described in \S \ref{secintro}. The restriction of $\varphi$ to the quotient space $V_\varphi/W$ is called the \emph{anisotropic part} of $\varphi$, and is denoted $\anispart{\varphi}$. Note that we have 
$$ \mathrm{dim}(\anispart{\varphi}) =  \mathrm{dim}(\varphi) - \witti{0}{\varphi} $$
(as opposed to the more familiar formula $\mathrm{dim}(\anispart{\varphi}) = \mathrm{dim}(\varphi) -2 \witti{0}{\varphi}$ from the theory of non-degenerate quadratic forms).

\subsection{The representation theorem} We write $D(\varphi)$ for the set $\lbrace \varphi(v)\;|\;v \in V_\varphi \rbrace$ of all elements of $F$ which are represented by $\varphi$. In this setting, the Cassels-Pfister representation theorem (\cite[Thm. 17.12]{EKM}) can be restated as follows:

\begin{proposition}[{see \cite[Prop. 2.6]{Hoffmann2}}] \label{propsubformtheorem} Let $\psi$ and $\varphi$ be quasilinear quadratic forms over $F$. Then $\anispart{\psi} \subset \anispart{\varphi}$ if and only if $D(\psi) \subset D(\varphi)$. In particular, $\anispart{\psi} \simeq \anispart{\varphi}$ if and only if $D(\psi) = D(\varphi)$.  \end{proposition} 

\subsection{Quasi-Pfister forms and quasi-Pfister neighbours} Let $n$ be a positive integer. We say that $\varphi$ is an \emph{$n$-fold quasi-Pfister form} if $\varphi \simeq \pfister{a_1,\hdots,a_n} :=\bigotimes_{i=1}^n \form{1,a_i}$ for some $a_1,\hdots,a_n \in F$. In this case, we have $\mathrm{dim}(\varphi) = 2^n$. For completeness, a $0$-fold quasi-Pfister form is a 1-dimensional form isomorphic to $\form{1}$. We say that $\varphi$ is a \emph{quasi-Pfister neighbour} if $\varphi$ is similar to a subform of a quasi-Pfister form $\pi$ with $\mathrm{dim}(\pi) < 2\mathrm{dim}(\varphi)$. 

\subsection{The norm form and norm degree} \label{secnormform} If $\varphi$ is non-zero, then the \emph{norm field} of $\varphi$, denoted $N(\varphi)$, is the smallest subfield of $F$ containing the set $\lbrace ab \;|\;a,b \in D(\varphi) \rbrace$. Explicitly, if $\varphi \simeq \form{a_0,a_1,\hdots,a_n}$ with $a_0 \neq 0$, then $N(\varphi) = F^2(a_0a_i\;|\;1\leq i \leq n)$. In particular, $N(\varphi)$ is a finite extension of $F^2$. The degree of this extension is called the \emph{norm degree} of $\varphi$, and is denoted $\mathrm{ndeg}(\varphi)$. Clearly $\mathrm{ndeg}(\varphi)$ is a power of $2$, and we write $\mathrm{lndeg}(\varphi)$ for the integer $\mathrm{log}_2\big(\mathrm{ndeg}(\varphi)\big)$. Up to isometry, there exists a unique anisotropic quasilinear quadratic form $\normform{\varphi}$ over $F$ such that $D(\normform{\varphi}) = N(\varphi)$ (existence is clear; uniqueness holds by Proposition \ref{propsubformtheorem}). The form $\normform{\varphi}$ is called the \emph{norm form} of $\varphi$, and it is easy to see that $\normform{\varphi}$ is in fact a quasi-Pfister form of dimension $2^{\mathrm{lndeg}(\varphi)}$. If $\varphi$ is anisotropic, then it is similar to a subform of $\normform{\varphi}$ by Proposition \ref{propsubformtheorem}. In fact, in this case, an anisotropic quasi-Pfister form over $F$ contains a subform similar to $\varphi$ if and only if it contains a subform isomorphic to $\normform{\varphi}$. This readily implies the following basic result:

\begin{lemma}[{see \cite[Cor. 3.11]{Scully1}}] \label{lemnormdegree} Let $\varphi$ be an anisotropic quasilinear quadratic form of dimension $\geq 1$ over $F$, and let $s \geq -1$ be the unique integer such that $2^s < \mathrm{dim}(\varphi) \leq 2^{s+1}$. Then $\mathrm{lndeg}(\varphi) \geq s+1$, and equality holds if and only if $\varphi$ is a quasi-Pfister neighbour.
\end{lemma}

\subsection{Similarity factors} We write $G(\varphi)^*$ for the set $\lbrace a \in F^*\;|\; a\varphi\simeq \varphi \rbrace$ of all \emph{similarity factors} of $\varphi$. The basic result on similarity factors in the quasilinear setting is the following:

\begin{proposition}[{see \cite[Lem. 6.3]{Hoffmann2}}] \label{propsimilarity} Let $\varphi$ be an anisotropic quasilinear quadratic form over $F$ and let $a \in F \setminus F^2$. Then the following are equivalent:
\begin{enumerate} \item $a \in G(\varphi)^*$.
\item $aD(\varphi) \subset D(\varphi)$.
\item $\varphi$ is divisible by $\pfister{a}$. 
\item $\witti{0}{\varphi_{F(\sqrt{a})}} = \frac{\mathrm{dim}(\varphi)}{2}$. 
\end{enumerate}\end{proposition}

\subsection{Function fields of quasilinear quadrics} \label{secff} We say that $\varphi$ is \emph{split} if $\mathrm{dim}(\anispart{\varphi})\leq 1$. If $\varphi$ is not split, then the projective $F$-quadric $\lbrace \varphi = 0 \rbrace$ is integral (see \cite[Lem. 7.1]{Hoffmann2}), and we write $F(\varphi)$ for its function field. In this case, we also write $F[\varphi]$ for the function field of the \emph{affine} $F$-quadric of the same equation. Clearly $F[\varphi]$ is $F$-isomorphic to a degree-$1$ purely transcendental extension of $F(\varphi)$. Explicitly, if $\varphi \simeq \form{a_0,a_1,\hdots,a_n}$ with $a_0 \neq 0$, then we have an $F$-isomorphism

$$ F[\varphi] \simeq F(X)\left(\sqrt{a_0^{-1}(a_1X_1^2 + \cdots + a_nX_n^2)}\right),$$
where $X = (X_1,\hdots,X_n)$ is an $n$-tuple of algebraically independent variables over $F$. The form $\varphi$ is evidently isotropic after scalar extension to $F(\varphi)$ (or $F[\varphi]$). 

\subsection{The Knebusch splitting tower} \label{secknebusch} We adapt an important construction of Knebusch (see \cite[\S 25]{EKM}) to the quasilinear setting: Let $F_0 = F$, $\varphi_0 = \anispart{\varphi}$, and inductively define
\begin{itemize} \item $F_r = F_{r-1}(\varphi_{r-1})$ (provided $\varphi_{r-1}$ is not split), and
\item $\varphi_r = \anispart{(\varphi_{F_r})}$ (provided $F_r$ is defined). \end{itemize}
Since the dimensions of the $\varphi_r$ become progressively smaller, this is a finite process, terminating at the first integer $h(\varphi)$ for which $\varphi_{h(\varphi)}$ is split; $h(\varphi)$ is called the \emph{height} of $\varphi$, and the tower of fields $F =F_0 \subset F_1 \subset \cdots \subset F_{h(\varphi)}$ is called the \emph{Knebusch splitting tower} of $\varphi$. This construction can be used to characterize quasi-Pfister neighbours as follows:

\begin{lemma}[{see \cite[\S 8]{HoffmannLaghribi1} or \cite[Cor. 3.11]{Scully1}}] \label{lemPNcriterion} Let $\varphi$ be an anisotropic quasilinear quadratic form of dimension $\geq 2$ over $F$. Then $\varphi$ is a quasi-Pfister neighbour if and only if $\varphi_1$ is similar to a quasi-Pfister form.
\end{lemma}

If $\varphi$ is \emph{not} split, then for each $1 \leq r \leq h(\varphi)$, we set $\mathfrak{i}_r(\varphi) = \witti{0}{\varphi_{F_r}} - \witti{0}{\varphi_{F_{r-1}}}$; $\witti{r}{\varphi}$ is called the $r$-th \emph{higher isotropy index} of $\varphi$. If $\varphi$ is anisotropic and $s$ is the unique non-negative integer such that $2^s < \mathrm{dim}(\varphi) \leq 2^{s+1}$, then it is known that $\witti{1}{\varphi} \leq \mathrm{dim}(\varphi) - 2^s$ (see \cite[Lem. 4.1]{HoffmannLaghribi2} or Theorem \ref{thmdivisibilityindices} below). If equality holds here, then we say that $\varphi$ has \emph{maximal splitting}. By Lemma \ref{lemPNcriterion}, anisotropic quasi-Pfister neighbours have maximal splitting. The converse is not true, but we have the following assertion, the analogue of which is still open for non-degenerate quadratic forms (even in characteristic $\neq 2$):

\begin{theorem}[{see \cite[Thm. 9.6]{Scully0}}] \label{thmmaxsplitting} Let $\varphi$ be an anisotropic quasilinear quadratic form of dimension $\geq 2$ over $F$ and let $s$ be the unique non-negative integer such that $2^s < \mathrm{dim}(\varphi) \leq 2^{s+1}$. If $\varphi$ has maximal splitting and $\mathrm{dim}(\varphi) > 2^s + 2^{s-2}$, then $\varphi$ is a quasi-Pfister neighbour. 
\end{theorem}

Finally, we will need to recall the evolution of the norm degree of a given form as one runs over its Knebusch splitting tower:

\begin{lemma}[{see \cite[Lem. 7.12]{Hoffmann2}}] \label{lemndegtower} Let $\varphi$ be a non-split quasilinear quadratic form over $F$. Then $\mathrm{lndeg}(\varphi_1) = \mathrm{lndeg}(\varphi) - 1$. \end{lemma}

\subsection{Divisibility indices} \label{secdivindices} If $\varphi$ is non-zero, then the \emph{divisibility index} of $\varphi$ is defined as the largest integer $\mathfrak{d}_0(\varphi)$ such that $\varphi$ is divisible by an anisotropic quasi-Pfister form of dimension $2^{\mathfrak{d}_0(\varphi)}$. We will need the following lemma, which is analogous to a well-known fact from the non-degenerate theory (cf. \cite[Cor. 2.1.11]{Kahn}):

\begin{lemma} \label{lemdivindexineq} Let $\varphi$ and $\sigma$ be quasilinear quadratic forms over $F$. If $\varphi$ is divisible by $\sigma$, then $\mathfrak{d}_0(\anispart{(\varphi)}) \geq \mathfrak{d}_0(\sigma)$. 
\begin{proof} Let $\pi$ be an anisotropic quasi-Pfister form which divides $\sigma$. Then $\pi$ also divides $\varphi$, and hence $\anispart{\varphi}$ by \cite[Prop. 4.19]{Hoffmann2}. \end{proof} \end{lemma}

 If $\varphi$ is \emph{not} split, then for each $1 \leq r \leq h(\varphi)$, we set $\mathfrak{d}_r(\varphi) = \mathfrak{d}_0(\varphi_r)$; $\mathfrak{d}_r(\varphi)$ is called the $r$-th \emph{higher divisibility index} of $\varphi$. The fundamental result concerning the higher divisibility indices of quasilinear quadratic forms is the first part of the following theorem:

\begin{theorem}[{see \cite[Thm. 6.1]{Scully1}}] \label{thmdivisibilityindices} Let $\varphi$ be an anisotropic quasilinear quadratic form of dimension $\geq 2$ over $F$ and let $1 \leq r \leq h(\varphi)$. Then
\begin{enumerate} \item $\mathfrak{i}_r(\varphi) \leq 2^{\mathfrak{d}_r(\varphi)}$.
\item If $\varphi_{r-1}$ is not similar to a quasi-Pfister form, then $\mathfrak{i}_r(\varphi)$ is divisible by $2^{\mathfrak{d}_{r-1}(\varphi)}$. \end{enumerate}
\end{theorem}

\begin{remark} Theorem \ref{thmdivisibilityindices} effectively determines all possible values of the Knebusch splitting pattern for quasilinear quadratic forms -- see \cite[\S 6]{Scully1}. In particular, it includes the quasilinear analogue of Karpenko's theorem on the possible values of $\mathfrak{i}_1(-)$ (\cite{Karpenko1}). Note that the problem of determining the possible values of the Knebusch splitting pattern for non-degenerate quadratic forms is wide open (even in characteristic $\neq 2$). 

\end{remark}

\section{The invariant $d(-)$}

Our approach to the quasilinear case of Conjecture \ref{mainconj} will be based on a certain procedure which permits to lower not only the value of $\mathrm{dim}(p)$, but also the value of $k$. In this section, we state and prove a technical lemma which will be needed to meet the second purpose. In order to improve the readability of what follows, it will be convenient to introduce notation which allows us to explicitly express the integer $k$ as a function of the pair $(p,q)$. With this in mind, we make the following definition:

\begin{definition} Given a quadratic form $\varphi$ over a field (of any characteristic), we set
$$ d(\varphi) = \mathrm{dim}(\varphi) - 2\witti{0}{\varphi}. $$ \end{definition}

Observe that if $\varphi$ is non-degenerate (or even regular), then $d(\varphi)$ is nothing else but the dimension of the anisotropic part of $\varphi$. In general, however, this need not be the case. For example, if $\varphi$ is quasilinear, then the statement only holds in the trivial case where $\varphi$ is already anisotropic. In fact, $d(-)$ can take negative values for sufficiently degenerate forms. We do, however, have the following basic observation:

\begin{lemma} \label{lempositived} If $\psi$ and $\varphi$ are anisotropic quadratic forms of dimension $\geq 2$ over a field $K$, then $d(\varphi_{K(\psi)}) \geq 0$. 
\begin{proof} The claim is that $\witti{0}{\varphi_{K(\psi)}} \leq \frac{\mathrm{dim}(\varphi)}{2}$. By \cite[Prop. 7.29 and Prop. 7.31]{EKM}, we have an orthogonal decomposition $\varphi \simeq \varphi' \perp \varphi''$ in which $\varphi'$ is non-degenerate and $\varphi''$ is quasilinear (if $\mathrm{char}(K) \neq 2$, this means that $\varphi'' = 0$). Moreover, it follows from \cite[Lem. 8.10]{EKM} that
$$ \witti{0}{\varphi_{K(\psi)}} \leq \frac{\mathrm{dim}(\varphi')}{2} + \witti{0}{\varphi''_{K(\psi)}}.$$
The problem is therefore reduced to the case where $\varphi$ is quasilinear. If $\psi$ is also quasilinear, then a proof of the needed claim can be found in \cite[Lem. 2.31 (1)]{Scully1}. If $\psi$ is not quasilinear, then $K(\psi)/K$ is a separably generated extension (i.e., the quadric defined by the vanishing of $\psi$ is generically smooth). In particular, we have $\witti{0}{\varphi_{K(\psi)}} = 0$ by \cite[Prop. 5.3]{Hoffmann2}, and so the statement holds trivially in this case. \end{proof} \end{lemma}

\begin{remark} Note that, in the statement of Conjecture \ref{mainconj}, the integer $k$ is nothing else but $d(q_{F(p)})$. The lemma therefore confirms that $k \geq 0$. \end{remark}

We now specialize to the quasilinear setting. As alluded to in the proof of Lemma \ref{lempositived}, anisotropic quasilinear quadratic forms remain anisotropic under separably generated field extensions (\cite[Prop. 5.3]{Hoffmann2}). In particular, we have:

\begin{lemma} \label{lemdsep} Let $\varphi$ be a quasilinear quadratic form over $F$. If $L$ is a separably generated field extension of $F$, then $d(\varphi_L) = d(\varphi)$. \end{lemma}

For later use, we will also need to observe the following interaction between $d(-)$ and the divisibility index invariant $\mathfrak{d}_0(-)$ (see \S \ref{secdivindices} above):

\begin{lemma} \label{lemdivisibilityofd} Let $\varphi$ and $\psi$ be anisotropic quasilinear quadratic forms of dimension $\geq 2$ over $F$. Then $d(\varphi_{F(\psi)})$ is divisible by $2^{\mathfrak{d}_0(\varphi)}$. 
\begin{proof} Let $\varphi \simeq \pi \otimes \sigma$, with $\pi$ a quasi-Pfister form. We need to show that $\mathrm{dim}(\pi)$ divides $d(\varphi_{F(\psi)})$. There are two cases to consider:\\

\noindent {\it Case 1.} If $\pi_{F(\psi)}$ is isotropic, then (since $\pi$ is a quasi-Pfister form) we have $\witti{0}{\pi_{F(\psi)}} = \frac{\mathrm{dim}(\pi)}{2}$ (see \cite[Cor. 7.9]{Hoffmann2}). In particular, we have
$$ \witti{0}{\varphi_{F(\psi)}} \geq \frac{\mathrm{dim}(\pi)}{2} \mathrm{dim}(\sigma) = \frac{\mathrm{dim}(\varphi)}{2}. $$
By Lemma \ref{lempositived}, it follows that $d(\varphi_{F(\psi)}) = 0$, and so the statement holds trivially. \\

\noindent {\it Case 2.} If $\pi_{F(\psi)}$ is anisotropic, then it is also a divisor of $\anispart{(\varphi_{F(\psi)})}$ by \cite[Prop. 4.19]{Hoffmann2}. Since
\begin{eqnarray*} d(\varphi_{F(\psi)}) &=& \mathrm{dim}(\varphi) - 2\witti{0}{\varphi_{F(\psi)}}\\
              &=& 2 \mathrm{dim}\big(\anispart{(\varphi_{F(\psi)})}\big) - \mathrm{dim}(\varphi), \end{eqnarray*}
(recall that $\mathrm{dim}(\anispart{\eta}) = \mathrm{dim}(\eta) - \witti{0}{\eta}$ in the quasilinear setting) the claim again holds. \end{proof} \end{lemma}

Our goal now is to prove Lemma \ref{RefinedDlem} below. We first make two simple observations:

\begin{lemma} \label{Dlem} Let $\sigma$, $\tau$ and $\phi$ be quasilinear quadratic forms over $F$ such that $\sigma \subset \tau \subset \varphi$. Let $a= \mathrm{dim}(\tau) - \mathrm{dim}(\sigma)$ and $b = \mathrm{dim}(\varphi) - \mathrm{dim}(\tau)$. Then
\begin{enumerate} \item $d(\sigma) \leq d(\tau) + a$, and
\item $d(\varphi) \leq d(\tau) + b$. \end{enumerate}
\begin{proof} (1) Since $\sigma$ is a codimension-$a$ subform of $\tau$, $\witti{0}{\sigma} \geq \witti{0}{\tau} - a$ (intersect the subspace of all isotropic vectors in $V_\tau$ with the image of $V_\sigma$ under the embedding $\sigma \subset \tau$). Thus:
\begin{eqnarray*} d(\sigma) &=& \mathrm{dim}(\sigma) - 2\witti{0}{\sigma}\\
                            &=& \mathrm{dim}(\tau) -a -2 \witti{0}{\sigma} \\
                            &\leq & \mathrm{dim}(\tau) - a - 2(\witti{0}{\tau} -a) \\
                            &=& \big(\mathrm{dim}(\tau) - 2\witti{0}{\tau}\big) + a\\
                            &=& d(\tau) + a. \end{eqnarray*}
(2) Since $\tau$ is a subform of $\varphi$, we obviously have $\witti{0}{\varphi} \geq \witti{0}{\tau}$. Hence
\begin{eqnarray*} d(\varphi) &=& \mathrm{dim}(\varphi) - 2\witti{0}{\varphi}\\
                            &=& \mathrm{dim}(\tau) +b  -2 \witti{0}{\varphi} \\
                            &\leq & \mathrm{dim}(\tau) + b - 2(\witti{0}{\tau}) \\
                            &=& \big(\mathrm{dim}(\tau) - 2\witti{0}{\tau}\big) + b\\
                            &=& d(\tau) + b, \end{eqnarray*}
 as claimed.\end{proof} \end{lemma}
 
\begin{lemma} \label{lemtensorinclusion} Let $\sigma$ and $\nu$ be anisotropic quasilinear quadratic forms over $F$ and let $\varphi = \anispart{(\sigma \otimes \nu)}$.
\begin{enumerate} \item If $a \in D(\sigma)$ is non-zero, then $a\nu \subset \varphi$.
\item If $b \in D(\nu)$ is non-zero, then $b \sigma \subset \varphi$. \end{enumerate}
\begin{proof} By symmetry, it suffices to prove (1). But if $a$ is a non-zero element of $D(\sigma)$, then
$$ D(a\nu) = aD(\nu) \subset D(\sigma \otimes \nu) = D(\varphi),$$
so the claim follows from Proposition \ref{propsubformtheorem}.\end{proof} \end{lemma}

We can now prove the following:

\begin{lemma} \label{RefinedDlem} Let $\sigma$ and $\nu$ be anisotropic quasilinear quadratic forms over $F$ with $\mathrm{dim}(\nu) \geq 2$, let $\varphi= \anispart{(\sigma \otimes \nu)}$ and let $l = \mathrm{dim}(\varphi) - \mathrm{dim}(\sigma)$. Then $l\geq 0$, and there exists an integer $0 \leq i \leq l$ such that
\begin{enumerate} \item $d(\sigma_{F(\nu)}) \leq i$, and
\item $d(\varphi_{F(\nu)}) \leq l - i$. 
\end{enumerate}
\begin{proof} Multiplying $\nu$ by a scalar if necessary, we can assume that $1 \in D(\nu)$. We can then write $\nu \simeq \form{1,a_1,\hdots,a_n}$ for some $a_1,\hdots,a_n \in F$. Let $X = (X_1,\hdots,X_n)$ be an $n$-tuple of algebraically independent variables over $F$, and set $K = F(X)$ and $L = K(\sqrt{\nu'(X)})$, where $\nu'(X) = a_1X_1^2 + \cdots + a_nX_n^2$. The field $L$ is $F$-isomorphic to the function field $F[\nu]$ of the affine quadric defined by the vanishing of $\nu$ (see \S \ref{secff} above). Now, since $1 \in D(\nu)$, we have $\sigma \subset \varphi$ by Lemma \ref{lemtensorinclusion} (2). In particular, we have $l \geq 0$. We now claim that there exists a form $\tau$ over $K$ such that $\sigma_K \subset \tau \subset \varphi_K$ and $d(\tau_L) = 0$. If true, this implies (via Lemma \ref{Dlem}) that
\begin{enumerate} \item[(1')] $d(\sigma_L) \leq i$, and
\item[(2')] $d(\varphi_L) \leq l - i$ \end{enumerate}
for some $0 \leq i \leq l$. This is equivalent to the assertion of the lemma, since (a) $L$ is $F$-isomorphic to $F[\nu]$, (b) $F[\nu]$ is $F$-isomorphic to a degree-1 purely transcendental extension of $F(\nu)$, and (c) $d(-)$ is invariant under purely transcendental base change (Lemma \ref{lemdsep}).

It thus remains to prove the existence of the form $\tau$. Let $\sigma_K \subset \tau \subset \varphi_K$. By definition, we have $d(\tau_L) = 0$ if and only if $\witti{0}{\tau_L} = \frac{\mathrm{dim}(\tau)}{2}$. Now $\tau$ is anisotropic (because $\varphi_K$ is), and we have $L = K(\sqrt{\nu'(X)})$. Thus, by Proposition \ref{propsimilarity}, our task is to show that $\tau$ can be chosen so that $\nu'(X) D(\tau) \subset D(\tau)$. We do this as follows: Consider the set $S$ of all subforms $\eta$ of $\varphi_K$ which contain $\sigma_K$ as a subform and which have the property that $\nu'(X)D(\eta) \subset D(\varphi_K)$. This set is non-empty, since it contains $\sigma_K$ by the definition of $\varphi$. We take $\tau$ to be an element of maximal dimension in $S$. It now remains to show that $\nu'(X)D(\tau) \subset D(\tau)$ for this choice of $\tau$. Suppose, for the sake of contradiction, that this is not the case. Then there exists $a \in D(\tau)$ such that $\nu'(X)a \notin D(\tau)$. Let $\tau' = \tau \perp \form{\nu'(X)a}$. We will show that $\tau' \in S$, thus contradicting the maximality of $\tau$. Note first that since $\nu'(X)a \notin D(\tau)$, $\tau'$ is anisotropic. Moreover, since $\nu'(X)a \in D(\varphi_K)$, we have
$$ D(\tau') = D(\tau) + K^2\nu'(X)a \subset D(\varphi_K), $$
so that $\tau' \subset \varphi_K$ by Proposition \ref{propsubformtheorem}. Now $\tau'$ obviously contains $\sigma_K$ as a subform (since $\tau$ does), and so the final step is to check that $\nu'(X)D(\tau') \subset D(\varphi_K)$. But
\begin{eqnarray*} \nu'(X)D(\tau') &=& \nu'(X)D(\tau) + K^2 \nu'(X)^2 a\\
                                  &=& \nu'(X)D(\tau) + K^2 a, \end{eqnarray*}
and since both $\nu'(X)D(\tau)$ and $K^2a$ lie in $D(\varphi_K)$ (the former since $\tau \in S$, and the latter because $a \in D(\tau)$), the claim is proved. \end{proof} \end{lemma}

\section{A refinement of the main conjecture in the quasilinear case}

The quasilinear case of Conjecture \ref{mainconj} is known to hold in the extreme case where $k=0$. In fact, rather more is true in this situation. Recall from \S \ref{secnormform} above that to any anisotropic quasilinear quadratic form $p$ over $F$, we can associate an anisotropic quasi-Pfister form $\normform{p}$ which contains a subform similar to $p$. We write $\mathrm{lndeg}(p)$ for the integer $\mathrm{log}_2\big(\mathrm{dim}(\normform{p})\big)$. The following result is due (independently) to Hoffmann and Laghribi:

\begin{proposition}[{see \cite[Thm. 1.5]{Laghribi1}, \cite[Thm. 7.7]{Hoffmann2}}] \label{propquasihyp} Let $p$ and $q$ be anisotropic quasilinear quadratic forms of dimension $\geq 2$ over $F$. If $\witti{0}{q_{F(p)}} = \frac{\mathrm{dim}(q)}{2}$, then $q$ is divisible by $\normform{p}$. In particular $\mathrm{dim}(q)$ is divisible by $2^{\mathrm{lndeg}(p)}$. \end{proposition}

If $s$ is the unique non-negative integer such that $2^s < \mathrm{dim}(p) \leq 2^{s+1}$, then $\mathrm{lndeg}(p) \geq s+1$, with equality holding if and only if $p$ is a quasi-Pfister neighbour (Lemma \ref{lemnormdegree} above). The integer $\mathrm{lndeg}(p)$ thus measures how far $p$ is from being a quasi-Pfister neighbour. If $p$ is ``generic'', for example, $\mathrm{lndeg}(p)$ takes its largest possible value of $\mathrm{dim}(p) - 1$:

\begin{example} Consider the form $p = \form{X_0,X_1,\hdots,X_n}$ over the rational function field $F(X_0,X_1, \hdots,X_n)$. Then $\mathrm{lndeg}(p) = n$. Indeed, $\normform{p}$ is isomorphic to the (evidently anisotropic) quasi-Pfister form $\pfister{X_0X_1,X_0X_2,\hdots,X_0X_n}$ in this case. \end{example}

Motivated by Proposition \ref{propquasihyp}, we consider the following modification of our problem:

\begin{problem} \label{refinedprob} Let $p$ and $q$ be anisotropic quasilinear quadratic forms of dimension $\geq 2$ over $F$, and let $k = \mathrm{dim}(q) - 2 \witti{0}{q_{F(p)}}$. Is is true that
$$ \mathrm{dim}(q)  = a2^{\mathrm{lndeg}(p)} + \epsilon $$
for some non-negative integer $a$ and integer $-k \leq \epsilon \leq k$? \end{problem}

In view of the above discussion, a positive answer to this problem for the pair $(p,q)$ immediately implies that Conjecture \ref{mainconj} holds for the same pair of forms. Unlike Conjecture \ref{mainconj}, however, one can only expect a positive answer to Problem \ref{refinedprob} in certain situations. This is already clear from consideration of the case where $q=p$:

\begin{lemma} \label{lemp=q} The $q =p$ case of Problem \ref{refinedprob} has a positive answer if and only if $p$ is a quasi-Pfister neighbour.
\begin{proof} Suppose first that $\mathrm{dim}(p) = a2^{\mathrm{lndeg}(p)} + \epsilon$ for integers $a \geq 0$ and $-k \leq \epsilon \leq k$. Since $q=p$, we have
$$ k = \mathrm{dim}(p) -2\witti{1}{p} < \mathrm{dim}(p). $$
This implies that $a \geq 1$, so that
$$ \mathrm{dim}(p) \geq 2^{\mathrm{lndeg}(p)} - k > 2^{\mathrm{lndeg}(p)} - \mathrm{dim}(p). $$
In other words, we have
$$ \mathrm{dim}(p) > 2^{\mathrm{lndeg}(p)-1}. $$
By Lemma \ref{lemnormdegree}, this is exactly what it means for $p$ to be a quasi-Pfister neighbour. Conversely, if $p$ is a quasi-Pfister neighbour, then $p$ has maximal splitting (see \S \ref{secknebusch} above). The reader will readily observe that this simply means that $\mathrm{dim}(p) = 2^{\mathrm{lndeg}(p)} - k$.  \end{proof} \end{lemma}

Nevertheless, we can still hope for a positive answer in many interesting cases. In particular, our main result (Theorem \ref{mainmain}) is that if $2^s < \mathrm{dim}(p) \leq 2^{s-1}$, then Problem \ref{refinedprob} has a positive answer as long as $k \leq 2^{s-1}$. Taking Lemma \ref{lemp=q} into account, we might expect that this can be improved to $k < 2^{s-1} + 2^{s-2}$ (but not further\footnote{When $q=p$, the inequality $k < 2^{s-1} + 2^{s-2}$ implies that $p$ is a Pfister neighbour (Theorem \ref{thmmaxsplitting}). This need not hold, however, if $k = 2^{s-1} + 2^{s-2}$; see \cite[Ex. 7.31]{Hoffmann2}.}) An improvement which takes into account the precise value of $\mathrm{dim}(p)$ is given in Theorem \ref{thmadditional} below. 

We now give an example which shows that, in the quasilinear case, the statement of Conjecture \ref{mainconj} is in some sense optimal. We work here with the norm degree invariant to indicate that the same applies to all our results in the direction of Problem \ref{refinedprob}. 

\begin{example} \label{exoptimal} Let $p$ be an anisotropic quasilinear quadratic form of dimension $\geq 2$ over a field $E$ of characteristic $2$. Let $\varphi$ be the anisotropic part of $\normform{p}$ over $E(p)$. By \cite[Cor. 4.10 and Cor. 7.9]{Hoffmann2}, $\varphi$ is a quasi-Pfister form of dimension $2^{\mathrm{lndeg}(p)-1}$. Moreover, there exists a subform $\tau \subset \normform{p}$ such that $\varphi \simeq \tau_{E(p)}$ by \cite[Lemma 5.1]{Hoffmann2}.  

Choose a non-negative integer $a$, a non-negative integer $k < 2^{\mathrm{lndeg}(p)}$ and let $\epsilon = k - 2l$ for some integer $0 \leq l \leq k/2$. Note that we have $\epsilon + l \geq 0$. Let $X = (X_1,\hdots,X_{a + \epsilon + l})$ be an $(a + \epsilon + l)$-tuple of algebraically independent variables over $E$ and let $F = E(X)$. Let $\sigma$ be a codimension-$l$ subform of $\normform{p}$ which contains $\tau$, and consider 
$$ q = (\normform{p})_F \otimes \form{X_1,\hdots,X_{a-1}} \perp X_a \sigma_F \perp \form{X_{a+1},\hdots,X_{a + \epsilon + l}}. $$
Since $F/E$ is a purely transcendental extension, $(\normform{p})_F$ is anisotropic. It is then clear that $q$ is also anisotropic. Note that we have $\mathrm{dim}(q) = a2^{\mathrm{lndeg}(p)} + \epsilon$. We claim that $\mathrm{dim}(q) - 2\witti{0}{q_{F(p)}} = k$ (here we write $F(p)$ for $F(p_F)$). Recall that $\varphi$ denotes the anisotropic part of $(\normform{p})_{E(p)}$. By construction, $\varphi$ is also the anisotropic part of $\sigma_{E(p)}$. Indeed, since $\tau \subset \sigma \subset \normform{p}$, Proposition \ref{propsubformtheorem} implies that
$$ \varphi \simeq \tau_{E(p)} \subset \anispart{(\sigma_{E(p)})} \subset \anispart{(\normform{p}_{E(p)})} =\varphi, $$
whence the claim. Now, the form
$$\eta =  (\varphi_{F(p)} \otimes \form{X_1,\hdots,X_a}) \perp \form{X_{a+1},\hdots,X_{a + \epsilon + l}}$$
is obviously anisotropic, and is therefore equal to $\anispart{(q_{F(p)})}$ by the preceding discussion. In particular, we have
\begin{eqnarray*} \mathrm{dim}(q) - 2\witti{0}{q_{F(p)}} &=& 2\mathrm{dim}(\eta) - \mathrm{dim}(q)\\
&=& 2(a2^{\mathrm{lndeg}(p)-1} + \epsilon + l) - (a2^{\mathrm{lndeg}(p)} + \epsilon)\\
&=& \epsilon + 2l\\
&=& k, \end{eqnarray*}
as desired. Observe now that because $F$ is a purely transcendental extension of $E$, replacing $p$ by $p_F$ neither changes $\mathrm{lndeg}(p)$ nor the fact that $p$ is anisotropic (\cite[Prop. 5.3]{Hoffmann2}). Moreover, the initial pair $(E,p)$ can be chosen so that $\mathrm{lndeg}(p)$ takes any prescribed value $\geq \mathrm{log}_2\big(\mathrm{dim}(p)\big)$. The example therefore shows that as far as the quasilinear case of Conjecture \ref{mainconj} is concerned, the integers $\witti{0}{q_{F(p)}}$ and $\mathrm{dim}(p)$ cannot by themselves determine any further gaps in the possible values of $\mathrm{dim}(q)$.  \end{example}

In what follows, we will mainly consider the statement formulated in Problem \ref{refinedprob}, as opposed to the statement of Conjecture \ref{mainconj}. Before proceeding to the proofs of our main results, it will be useful to record the following trivial reformulation of the former:

\begin{lemma} \label{lemreformconj} Problem \ref{refinedprob} admits a positive answer if and only if there exists a non-negative integer $a$ such that
$$ \witti{0}{q_{F(p)}} \leq a2^{\mathrm{lndeg}(p) - 1} \leq \mathrm{dim}(q) - \witti{0}{q_{F(p)}}. $$
\begin{proof} For any integer $a$, we have
\begin{eqnarray*} & -k \leq \mathrm{dim}(q) - a2^{\mathrm{ndeg}(p)} \leq k \\
\Leftrightarrow & -\mathrm{dim}(q) + 2\witti{0}{q_{F(p)}} \leq \mathrm{dim}(q) - a2^{\mathrm{ndeg}(p)} \leq \mathrm{dim}(q) - 2\witti{0}{q_{F(p)}} \\
\Leftrightarrow & 2 \witti{0}{q_{F(p)}} \leq 2(\mathrm{dim}(q) - a2^{\mathrm{ndeg}(p)-1}) \leq 2(\mathrm{dim}(q) - \witti{0}{q_{F(p)}}) \\
\Leftrightarrow & \witti{0}{q_{F(p)}} \leq a2^{\mathrm{lndeg}(p) - 1} \leq \mathrm{dim}(q) - \witti{0}{q_{F(p)}}, \end{eqnarray*} 
whence the claim. \end{proof} \end{lemma}

\begin{remarks} \label{remsrefinedconj} \begin{enumerate}[leftmargin=*] \item By Lemma \ref{lemndegtower}, the integer $\mathrm{lndeg}(p) - 1$ appearing as the exponent of the $2$-power here is equal to $\mathrm{lndeg}(p_1)$. We will use this fact in the sequel.
\item It follows from \cite[Prop. 5.3]{Hoffmann2} that the answer to Problem \ref{refinedprob} is invariant under replacing $F$ with a separably generated (e.g., purely transcendental) extension of itself.  \end{enumerate} \end{remarks}

\section{Main results} \label{secmainresults}

We are now ready to prove our main results towards the quasilinear case of Conjecture \ref{mainconj}. Our approach rests on the following key result from \cite{Scully2}:

\begin{theorem}[{\cite[Thm. 6.6]{Scully2}}]\label{highersubformthm} Let $\psi$ and $\varphi$ be anisotropic quasilinear quadratic forms of dimension $\geq 2$ over $F$ such that $\varphi_{F(\psi)}$ is isotropic. Then, after possibly replacing $F$ with a purely transcendental extension of itself $F$, there exists an anisotropic quasilinear quadratic form $\tau$ over $F$ such that
$$ \anispart{(\tau \otimes \psi_1)} \subset \anispart{(\varphi_{F(\psi)})} \hspace{1cm} \text{and} \hspace{1cm} \mathrm{dim}(\tau) \geq \witti{0}{\varphi_{F(\psi)}}.$$ \end{theorem}

As shown in \cite[\S 6]{Scully2}, Theorem \ref{highersubformthm} has many applications in the theory of quasilinear quadratic forms. The present article represents another demonstration of its range. First, we observe that it immediately implies the following:

\begin{proposition}\label{PNprop} Problem \ref{refinedprob} has a positive answer if $p$ is a quasi-Pfister neighbour. In particular, the quasilinear case of Conjecture \ref{mainconj} holds when $p$ is a quasi-Pfister neighbour.
\begin{proof} Let $s$ be the unique non-negative integer such that $2^s < \mathrm{dim}(p) \leq 2^{s+1}$. Since $p$ is a quasi-Pfister neighbour, we have $\mathrm{lndeg}(\varphi) = s+1$. If $q_{F(p)}$ is anisotropic then conjecture holds trivially, so let us suppose otherwise. By Theorem \ref{highersubformthm} and Remark \ref{remsrefinedconj} (2), we may then assume that there exists an anisotropic form $\tau$ over $F(p)$ such that $\mathrm{dim}(\tau) = \witti{0}{q_{F(p)}}$ and $\anispart{(\tau \otimes p_1)} \subset \anispart{(q_{F(p)})}$. On the other hand, $\tau$ is similar to a subform of $\anispart{(\tau \otimes p_1)}$ by Lemma \ref{lemtensorinclusion} (2). We therefore have inequalities
$$ \mathrm{dim}(\tau) \leq \mathrm{dim}\big(\anispart{(\tau \otimes p_1)}\big) \leq \mathrm{dim}\big(\anispart{(q_{F(p)})}\big),$$
or, in other words,
$$ \witti{0}{q_{F(p)}} \leq \mathrm{dim}\big(\anispart{(\tau \otimes p_1)}\big) \leq \mathrm{dim}(q) - \witti{0}{q_{F(p)}}. $$
Now, since $p$ is a quasi-Pfister neighbour, $p_1$ is similar to an $s$-fold Pfister form (Lemma \ref{lemPNcriterion} above). By \cite[Prop. 4.19]{Hoffmann2}, it follows that $\anispart{(\tau \otimes p_1)}$ is divisible by $p_1$. In particular, $\mathrm{dim}\big(\anispart{(\tau \otimes p_1)}\big)$ is divisible by $2^s$. Thus, by the preceding discussion, there exists a positive integer $a$ such that
$$ \witti{0}{q_{F(p)}} \leq a2^s \leq \mathrm{dim}(q) - \witti{0}{q_{F(p)}}.$$
By Lemma \ref{lemreformconj} and the fact that $\mathrm{lndeg}(p) = s+1$, this is exactly what we wanted. \end{proof} \end{proposition}

Now, while the general case seems to be more complicated, Theorem \ref{highersubformthm} at least permits us to set up an inductive approach to Problem \ref{refinedprob}. The basic step is the following lemma:

\begin{lemma} \label{leminductionlemma} Let $\psi$ be an anisotropic quasilinear quadratic form of dimension $\geq 2$ over a field $L$ of characteristic $2$, let $0 < m < 2^{\mathrm{lndeg}(\psi) - 2}$, and let $m'$ be the largest integer strictly less than $m$ which is divisible by $2^{\mathfrak{d}_1(\psi)}$. If Problem \ref{refinedprob} has a positive answer when $F = L(\psi)$, $p=\psi_1$ and $k \leq m'$, then it also has a positive answer when $F=L$, $p =\psi$ and $k \leq m$.

\begin{remark} We remind the reader that $\mathfrak{d}_1(\psi)$ denotes the largest integer $r$ such that $\psi_1$ is divisible by an anisotropic $r$-fold quasi-Pfister form (see \S \ref{secdivindices} above).
\end{remark}

\begin{proof} To simplify the notation, let $L_1 = L(\psi)$. We make the stated assumption regarding Problem \ref{refinedprob} in the case where $F = L_1$ and $p =\psi_1$. Let $q$ be an anisotropic quasilinear quadratic form of dimension $\geq 2$ over $L$ such that $k = d(q_{L_1}) \leq m$. Our aim is to show that $\mathrm{dim}(q) =a2^{\mathrm{lndeg}(p)} + \epsilon$ for some integers $a \geq 0$ and $-k\leq  \epsilon \leq k$. In view of Proposition \ref{propquasihyp}, we can assume that $k>0$. If $q_{L_1}$ is anisotropic, then the statement holds trivially. If not, then Theorem \ref{highersubformthm} and Remark \ref{remsrefinedconj} (2) allow us to assume that there exists an anisotropic form $\tau$ over $L_1$ such that $\mathrm{dim}(\tau) = \witti{0}{q_{L_1}}$ and $\anispart{(\tau \otimes \psi_1)} \subset \anispart{(q_{L_1})}$. Now, by Lemma \ref{lemreformconj} and Remark \ref{remsrefinedconj} (1), proving our assertion amounts to showing that
$$ \witti{0}{q_{L_1}} \leq a2^{\mathrm{lndeg}(\psi_1)} \leq \mathrm{dim}(q) - \witti{0}{q_{L_1}}$$
for some non-negative integer $a$, or, equivalently, that
\begin{equation} \label{equation1} \mathrm{dim}(\tau) \leq a2^{\mathrm{lndeg}(\psi_1)} \leq \mathrm{dim}(\anispart{(q_{L_1})}) \end{equation} for some non-negative integer $a$. We will work with the latter formulation. Let $\varphi = \anispart{(\tau \otimes \psi_1)}$ and let $l = \mathrm{dim}(\varphi) - \mathrm{dim}(\tau)$. By Lemma \ref{RefinedDlem}, we have that
\begin{enumerate} 
\item $d(\tau_{L_1(\psi_1)}) \leq i$, and
\item $d(\varphi_{L_1(\psi_1)}) \leq l - i$
\end{enumerate}
for some integer $0 \leq i \leq l$ (recall here that $l \geq 0$ because $\tau$ is similar to a subform of $\varphi$). We now have three cases to consider:\\

\noindent {\bf Case 1.} If $i=0$, then Proposition \ref{propquasihyp} tells us that $\mathrm{dim}(\tau)$ is divisible by $2^{\mathrm{lndeg}(\psi_1)}$. Thus, \eqref{equation1} certainly holds for some positive integer $a$ in this case. \\

\noindent {\bf Case 2.} If $i = l$, then Proposition \ref{propquasihyp} tells us that $\mathrm{dim}(\varphi)$ is divisible by $2^{\mathrm{lndeg}(\psi_1)}$. Since
$$ \mathrm{dim}(\tau) \leq \mathrm{dim}(\varphi) \leq \mathrm{dim}(\anispart{(q_{L_1})}), $$
\eqref{equation1} also holds for some positive integer $a$ in this case. \\

\noindent {\bf Case 3.} Suppose now that $0<i<l$. We claim that $d(\tau_{L_1(\psi_1)})$ and $d(\varphi_{L_1(\psi_1)})$ are both $\leq m'$ (recall that $m'$ is the largest integer $<m$ which is divisible by $2^{\mathfrak{d}_1(\psi)}$). First, we have
\begin{eqnarray*} l &=& \mathrm{dim}(\varphi) - \mathrm{dim}(\tau) \\
&\leq& \mathrm{dim}(\anispart{(q_{L_1})}) - \witti{0}{q_{L_1}} \\
&= & \mathrm{dim}(q) - 2\witti{0}{q_{L_1}} \\
&=& k\end{eqnarray*}
In particular, since $i>0$, we have that $d(\varphi_{L_1(\psi_1)})< l \leq k \leq m$. On the other hand $d(\varphi_{L_1(\psi_1)})$ is divisible by $2^{\mathfrak{d}_0(\varphi)}$ by Lemma \ref{lemdivisibilityofd}. Since $\mathfrak{d}_0(\varphi) \geq \mathfrak{d}_0(\psi_1) = \mathfrak{d}_1(\psi)$ (Lemma \ref{lemdivindexineq}), we conclude that $d(\varphi_{L_1(\psi_1)})$ is divisible by $2^{\mathfrak{d}_1(\psi)}$, and is therefore $\leq m'$. This argument also proves the claim about $d(\tau_{L_1(\psi_1)}) = i$. Indeed, if $i>m'$, then $l-i < 2^{\mathfrak{d}_1(\psi)}$ by the definition of $m'$ and the fact that $l \leq m$. Since $d(\varphi_{L_1(\psi_1)}) = l-i$ is divisible by $2^{\mathfrak{d}_1(\psi)}$, this implies that $l-i =0$, contradicting our assumption. Now, since $d(\tau_{L_1(\psi_1)})$ and $d(\varphi_{L_1(\psi_1)})$ are both $\leq m'$, we can invoke our initial hypothesis to get that
$$ \mathrm{dim}(\tau) = b2^{\mathrm{lndeg}(\psi_1)} + \epsilon_1 \hspace{1cm} \text{and} \hspace{1cm} \mathrm{dim}(\varphi) = c2^{\mathrm{lndeg}(\psi_1)} + \epsilon_2 $$
for some non-negative integers $b,c$ and some integers $-i \leq \epsilon_1 \leq i$ and $i-l \leq \epsilon_2 \leq l-i$. If $\epsilon_1 \leq 0$, then
\begin{eqnarray*} \mathrm{dim}(\tau) &\leq & b2^{\mathrm{lndeg}(\psi_1)} \\
                                     & \leq & \mathrm{dim}(\tau) + i \\
                                     & \leq & \mathrm{dim}(\tau) + l \\
                                     & = & \mathrm{dim}(\varphi)\\
                                     & \leq & \mathrm{dim}(\anispart{(q_{L_1})}), \end{eqnarray*}
and so \eqref{equation1} holds with $a = b$. Similarly, if $\epsilon_2 \geq 0$, then one immediately checks that \eqref{equation1} holds with $a = c$. Finally, suppose that $\epsilon_1 > 0$ and $\epsilon_2 < 0$. Since $\mathrm{dim}(\varphi)\geq \mathrm{dim}(\tau)$, we then have that $c>b$. In particular,
\begin{eqnarray*} 2^{\mathrm{lndeg}(\psi_1)} & \leq & (c-b)2^{\mathrm{lndeg}(\psi_1)}\\
               &=& \big(\mathrm{dim}(\varphi) - \mathrm{dim}(\tau)\big) + (\epsilon_1 - \epsilon_2)\\
               &\leq & l + \big(i + (l-i)\big) \\
               & =& 2l. \end{eqnarray*}
This is impossible, however. Indeed, we noted above that $l\leq m$, and since $m < 2^{\mathrm{lndeg}(\psi)-2}$, we have (using Lemma \ref{lemndegtower}) that
$$ 2l < 2^{\mathrm{lndeg}(\psi) -1} = 2^{\mathrm{lndeg}(\psi_1)}. $$
We conclude that this case cannot occur, and so the lemma is proved.  \end{proof}  \end{lemma}

Recall from \cite[\S 2.K]{Scully1} that the \emph{quasi-Pfister height} of a quasilinear quadratic form $\varphi$ is defined as the smallest non-negative integer $h_{\mathrm{qp}}(\varphi)$ such that $\varphi_{h_{\mathrm{qp}}(\varphi)}$ is similar to a quasi-Pfister form. The previous lemma now implies:

\begin{proposition} \label{propqphresult} Problem \ref{refinedprob} has a positive answer in the case where $k< \mathrm{dim}(p_{h_{\mathrm{qp}}(p)})$.

\begin{proof} Let $r = h_{\mathrm{qp}}(p)$. We argue by induction on $r$. If $r\leq 1$, then $p$ is a quasi-Pfister neighbour (Lemma \ref{lemPNcriterion}) and the statement holds by Proposition \ref{PNprop}. Assume now that $r\geq 2$. Then $\mathrm{lndeg}(p) \geq \mathrm{lndeg}(p_r) + 2$ by Lemma \ref{lemndegtower}. In particular, since $p_r$ is a quasi-Pfister form, we have
$$ k < \mathrm{dim}(p_r) = 2^{\mathrm{lndeg}(p_r)} \leq 2^{\mathrm{lndeg}(p) -2}. $$
In particular, by applying Lemma \ref{leminductionlemma} in the case where $L = F$ and $\psi = p$, we can achieve a reduction in $r$ without changing the assumption on $k$ (the exchange $p \rightarrow p_1$ does not change $\mathrm{dim}(p_{h_{\mathrm{qp}}(p)})$). The result follows.  \end{proof} \end{proposition}

The idea now is to examine Problem \ref{refinedprob} by repeatedly applying Lemma \ref{leminductionlemma} over the Knebusch splitting tower of $p$. In order to get something concrete out of this, however, we require some non-trivial information concerning the possible evolution of $p$ as we pass over this tower. Fortunately, such information is available in the form of Theorems \ref{thmmaxsplitting} and \ref{thmdivisibilityindices}, and this leads us to the following observation:

\begin{proposition} \label{propkeyprop} Let $s$ be a positive integer and let $0 \leq m <2^s$. Suppose that Problem \ref{refinedprob} has a positive answer whenever $k<m$. Then it also has a positive answer in the situation where $2^s < \mathrm{dim}(p) \leq 2^{s+1}$ and $k \leq \mathrm{min}(m+\mathrm{dim}(p) -2^s - 2^{s-2}, 2^s -1)$. 
\begin{proof} Suppose that $2^s < \mathrm{dim}(p) \leq 2^{s+1}$ and let $k \leq \mathrm{min}(m+\mathrm{dim}(p) -2^s - 2^{s-2}, 2^s -1)$. By the discussion in \S \ref{secknebusch}, there exists an integer $0 < r < h(p)$ such that $\mathrm{dim}(p_r) = 2^s$. If $p_r$ is a quasi-Pfister form, then $\mathrm{dim}(p_{h_{\mathrm{qp}}(p)}) \geq 2^s$. Since $k< 2^s$, the statement follows from Proposition \ref{propqphresult} in this case. Suppose now that $p_r$ is not a quasi-Pfister form. Then $\mathrm{lndeg}(p_r) \geq s+1$. In particular, if $0 \leq i < r$, then we have
$$ 2^{\mathrm{lndeg}(p_i) -2} \geq 2^{s+2 - 2} = 2^s > k $$
(see Lemma \ref{lemndegtower}). Thus, repeatedly applying Lemma \ref{leminductionlemma} over the Knebusch splitting tower $(F_i)$ of $p$ (specifically, consider the sequence $(F,p) \rightarrow (F_1,p_1) \rightarrow (F_2,p_2) \rightarrow \cdots \rightarrow (F_{r-1},p_{r-1})$), there exists a descending sequence of integers $m_1 > m_2 > \cdots > m_r$ such that
\begin{enumerate} \item $m_1 < k$.
\item For each $1 \leq i \leq r$, $m_i$ is divisible by $2^{\mathfrak{d}_i(p)}$.
\item If Problem \ref{refinedprob} has a positive answer for the triple $(F_r,p_r,m_r)$, then it also has a positive answer for the triple $(F,p,k)$. \end{enumerate}
In view of our initial hypothesis, it now only remains to show that $m_r < m$. Note first that $\mathfrak{d}_{j+1}(p) \geq \mathfrak{d}_{j}(p)$ for each $1 \leq j < r$ (see Theorem \ref{thmdivisibilityindices}). In particular, (2) implies that, for each such $j$, $m_j - m_{j-1}$ is a positive integer divisible by $2^{\mathfrak{d}_j(p)}$. Together with (1), we therefore have
$$ m_r = m_1 - \sum_{j=1}^{r-1} (m_j - m_{j+1})< k - \sum_{j=1}^{r-1} 2^{\mathfrak{d}_j(p)}.$$
We now invoke Theorem \ref{thmdivisibilityindices} (1), which tells us that $2^{\mathfrak{d}_j(p)} \geq \witti{j}{p}$ for all $j<r$. In particular, we have
\begin{eqnarray*} m_r &<& k - \sum_{j=1}^{r-1}\witti{j}{p} \\
 &=& k - \big(\mathrm{dim}(p) - \mathrm{dim}(p_{r-1})\big) \\
 & \leq & m + \mathrm{dim}(p) - 2^s - 2^{s-2} - \mathrm{dim}(p) + \mathrm{dim}(p_{r-1}) \\
 & = & m + \mathrm{dim}(p_{r-1}) - 2^{s} - 2^{s-2} \end{eqnarray*}
Thus, to complete the proof, it suffices to show that $\mathrm{dim}(p_{r-1}) \leq 2^s + 2^{s-2}$. To see that this is indeed true, note first that $p_{r-1}$ has maximal splitting (see \S \ref{secknebusch} above) by the very definition of $r$. On the other hand, since $p_r$ is not similar to a quasi-Pfister form, $p_{r-1}$ is \emph{not} a quasi-Pfister neighbour by Lemma \ref{lemPNcriterion}. The needed conclusion therefore follows from Theorem \ref{thmmaxsplitting}. \end{proof}\end{proposition}

An easy induction now yields our main result:

\begin{theorem} \label{mainmain} Let $p$ and $q$ be anisotropic quasilinear quadratic forms of dimension $\geq 2$ over $F$, let $s$ be the unique non-negative integer such that $2^s < \mathrm{dim}(p) \leq 2^{s+1}$, and let $k = \mathrm{dim}(q) - 2\witti{0}{q_{F(p)}}$. If $k \leq 2^{s-1}$, then
$$ \mathrm{dim}(q) = a2^{\mathrm{lndeg}(p)} + \epsilon $$
for some non-negative integer $a$ and integer $-k \leq \epsilon \leq k$. In particular, the quasilinear case of Conjecture \ref{mainconj} holds when $k \leq 2^{s-1}$. 
\begin{proof} We argue by induction on $\mathrm{dim}(p)$. If $\mathrm{dim}(p) \leq 2$, then the statement holds by Proposition \ref{propquasihyp}. Assume now that $\mathrm{dim}(p) \geq 3$. If $\mathrm{lndeg}(p) = s+1$, then $p$ is a quasi-Pfister neighbour (Lemma \ref{lemnormdegree}) and we can invoke Proposition \ref{PNprop}. We can therefore assume that $\mathrm{lndeg}(p) \geq s+2$ (again, see Lemma \ref{lemnormdegree}). In particular, we have
$$ k \leq 2^{s-1} < 2^{\mathrm{lndeg}(p) - 2}. $$
Thus, by Lemma \ref{leminductionlemma}, it suffices to prove the following: If $\varphi$ is an anisotropic quasilinear quadratic form over $L = F(p)$, and $k' = \mathrm{dim}(\varphi) - 2\witti{0}{\varphi_{L(p_1)}} < 2^{s-1}$ , then $\mathrm{dim}(\varphi) = b2^{\mathrm{lndeg}(p_1)} + \epsilon'$ for some non-negative integer $b$ and some integer $-k' \leq \epsilon' \leq k'$. Now, by the discussion of \S \ref{secknebusch}, we have $\mathrm{dim}(p_1) \geq 2^s$. If this inequality is strict, then we are done by the induction hypothesis. If $\mathrm{dim}(p_1) = 2^s$, the induction hypothesis at least tells us that the claim holds when $k' \leq 2^{s-2}$. But, by Proposition \ref{propkeyprop}, the claim then holds when
\begin{eqnarray*} k' &\leq & \mathrm{min}(2^{s-2} + 1 + \mathrm{dim}(p) - 2^{s-1} - 2^{s-3}, 2^{s-1}-1)\\
&=& \mathrm{min}(2^{s-2} + 1 + \mathrm{dim}(2^s) - 2^{s-1} - 2^{s-3}, 2^{s-1}-1) \\
&=& \mathrm{min}(2^{s-1} + 2^{s-3} + 1, 2^{s-1}-1)\\
&=& 2^{s-1} - 1, \end{eqnarray*}
and so the theorem is proved. \end{proof} \end{theorem}

Feeding this back into Proposition \ref{propkeyprop} now gives the following additional result in the direction of our conjecture:

\begin{corollary} \label{cortomain} Let $p$ and $q$ be anisotropic quasilinear quadratic forms of dimension $\geq 2$ over $F$, let $s$ be the unique non-negative integer such that $2^s < \mathrm{dim}(p) \leq 2^{s+1}$, and let $k = \mathrm{dim}(q) - 2\witti{0}{q_{F(p)}}$. If $k \leq \mathrm{min}(\mathrm{dim}(p) - 2^{s-1} - 2^{s-2}+1, 2^s-1)$, then
$$ \mathrm{dim}(q) = a2^{\mathrm{lndeg}(p)} + \epsilon $$
for some non-negative integer $a$ and integer $-k \leq \epsilon \leq k$. In particular, the quasilinear case of Conjecture \ref{mainconj} holds when $k \leq \mathrm{min}(\mathrm{dim}(p) - 2^{s-1} - 2^{s-2}+1, 2^s-1)$.

\begin{proof} By Theorem \ref{mainmain}, the result holds if $k < 2^{s-1} + 1$. Since
$$ 2^{s-1} + 1 + \mathrm{dim}(p) -2^s - 2^{s-2} = \mathrm{dim}(p) - 2^{s-1} - 2^{s-2}+1, $$
the claim now follows from Proposition \ref{propkeyprop}.\end{proof} \end{corollary}

In particular, we get that the quasilinear case of Conjecture \ref{mainconj} holds when $\mathrm{dim}(p)$ is close enough to $2^{s+1}$:

\begin{corollary} \label{corhighdimresult}  The quasilinear case of Conjecture \ref{mainconj} holds when 
$$ 2^{s+1}-2^{s-2} - 3 \leq \mathrm{dim}(p) \leq 2^{s+1}. $$
\begin{proof} It is enough to show that the statement of the conjecture holds when $k \leq 2^s - 2$ (Remark \ref{introrems} (2)). But if $\mathrm{dim}(p) \geq 2^{s+1} - 2^{s-2} - 3$, then
$$ \mathrm{dim}(p) - 2^{s-2} - 2^{s-2} + 1 \geq 2^{s} - 2, $$
and since $\mathrm{lndeg}(p) \geq s+1$, the claim follows from Corollary \ref{cortomain}. \end{proof} \end{corollary}

\begin{corollary} \label{cordimleq8} The quasilinear case of Conjecture \ref{mainconj} holds when $\mathrm{dim}(p) \leq 8$. 
\begin{proof} The dimension condition of the previous corollary clearly holds if $s \leq 3$. \end{proof} \end{corollary}

\section{Beyond the main result} \label{secbeyond}

The quasilinear case of Conjecture \ref{mainconj} remains open in the situation where $2^{s-1}<k<2^s$. Corollaries \ref{cortomain}, \ref{corhighdimresult} and \ref{cordimleq8} give some partial results in this direction. In this last section, we show that one can eliminate infinitely bad values of $\mathrm{dim}(q)$ without imposing any assumption on $k$. More precisely, we prove the following:

\begin{theorem} \label{thmadditional} The quasilinear case of Conjecture \ref{mainconj} holds if 
$$2^N \leq \mathrm{dim}(q) \leq 2^N + 2^{s+1} $$
for some positive integer $N$.  \end{theorem}

The proof is similar to that of Theorem \ref{mainmain}, but we will now also use the fact that Conjecture \ref{mainconj} is known to be true when $\mathrm{dim}(q) \leq 2^s$. In other words, we will use (the quasilinear case of) the Separation Theorem (\cite{Hoffmann1},\cite{HoffmannLaghribi2}). The following lemma is a trivial extension of the latter result:

\begin{lemma} \label{lemhoffmann} Let $p$ and $q$ be anisotropic quasilinear quadratic forms over dimension $\geq 2$ over $F$ and let $s$ be the unique non-negative integer such that $2^s < \mathrm{dim}(p) \leq 2^{s+1}$. Then $\witti{0}{q_{F(p)}} \leq \mathrm{max}(0,\mathrm{dim}(q) - 2^s)$. 
\begin{proof} Let $r = \mathrm{max}(0,\mathrm{dim}(q) - 2^s)$ and let $\varphi$ be a codimension $r$ subform of $q$. To prove the lemma, it suffices to show that $\varphi_{F(p)}$ is anisotropic (see Lemma \ref{Dlem} (1) above). But $\mathrm{dim}(\varphi) \leq 2^s < \mathrm{dim}(p)$, so this holds by the Separation Theorem (\cite[Thm. 1.1]{HoffmannLaghribi2}). \end{proof} \end{lemma}

As a consequence, we have:

\begin{lemma} \label{lemtrueforlow} The quasilinear case of Conjecture \ref{mainconj} holds if $\mathrm{dim}(q) \leq 2^{s+1} +k$.
\begin{proof} It suffices to consider the case where $\mathrm{dim}(q) < 2^{s+1} - k$. In this case, however, we have
\begin{eqnarray*} \witti{0}{q_{F(p)}} &=&  \frac{1}{2}\left(\mathrm{dim}(q) - k\right)\\
&>& \frac{1}{2}\left(2\mathrm{dim}(q) -2^{s+1}\right)\\
&=& \mathrm{dim}(q) -2^s,\end{eqnarray*}
so that $\witti{0}{q_{F(p)}} = 0$ by Lemma \ref{lemhoffmann}. Since the statement of the conjecture holds vacuously in this case, the lemma is proved. \end{proof} \end{lemma}

We also need the following fact which follows from standard specialization arguments:

\begin{lemma} \label{lem2s+1} To prove Conjecture \ref{mainconj}, it suffices to treat the case where $\mathrm{dim}(p) = 2^s + 1$. 
\begin{proof} Let $\sigma \subset p$ be a subform of dimension $2^s + 1$. We claim that the substitution $p \rightarrow \sigma$ does not increase the value of $k$. In other words, we claim that $\witti{0}{q_{F(\sigma)}} \geq \witti{0}{q_{F(p)}}$. By \cite[Lem. 3.4]{Scully0}, it suffices to show that there exists an $F$-place from $F(p)$ to $F(\sigma)$. Let $X_{\sigma}$ and $X_{p}$ be the projective quadrics defined by $\sigma$ and $p$, respectively. The inclusion $\sigma \subset p$ gives a regular embedding $X_{\sigma} \hookrightarrow X_p$. The quadric $X_p$ is then regular at the generic point of $X_\sigma$. The existence of the desired $F$-place then follows from \cite[Lem. A.4]{Scully0}. \end{proof} \end{lemma}

We are now ready to give the proof of Theorem \ref{thmadditional};

\begin{proof}[Proof of Theorem \ref{thmadditional}] By Lemmas \ref{lem2s+1} and \ref{lemtrueforlow}, we can assume that
\begin{enumerate} \item $\mathrm{dim}(p) = 2^s + 1$, and
\item $\mathrm{dim}(q) > 2^{s+1} + k$. 
\end{enumerate}
By (2), we then have that $N \geq s+1$. If $\mathrm{dim}(q) \leq 2^N + k$, there is nothing to prove, so suppose now that $2^N + k <\mathrm{dim}(q) \leq 2^N + 2^{s+1}$. By Lemma \ref{lempositived} and the definition of $k$, we then have that
$$ 2^s \leq 2^{N-1} < \witti{0}{q_{F(p)}} \leq 2^{N-1} + 2^s. $$ 

In particular, $q_{F(p)}$ is isotropic. Thus, by Theorem \ref{highersubformthm} and Remark \ref{remsrefinedconj} (2), we can assume that there exists an anisotropic form $\tau$ over $F(p)$ such that $\mathrm{dim}(\tau) = \witti{0}{q_{F(p)}}$ and $\anispart{(\tau \otimes p_1)} \subset \anispart{(q_{F(p)})}$. Let $\varphi = \anispart{(\tau \otimes p_1)}$. As in the proof of Lemma \ref{leminductionlemma}, we then have
$$ \witti{0}{q_{F(p)}} \leq \mathrm{dim}(\varphi) \leq \mathrm{dim}(q) - \witti{0}{q_{F(p)}}.$$
Since $\witti{0}{q_{F(p)}} \leq 2^{N-1} + 2^s$, it follows from Lemma \ref{lemreformconj} that it will be sufficient to show that $\mathrm{dim}(\varphi) \geq 2^{N-1} + 2^s$. Similar to the arguments of \S \ref{secmainresults}, the main point is to observe that $\varphi$ becomes isotropic to a considerable extent over the function field of $\tau$. More precisely, let $L = F(p)(\tau)$. As $\mathrm{dim}(p) = 2^s + 1$, we have $\mathrm{dim}(p_1) = 2^s$. Since $\mathrm{dim}(\tau) = \witti{0}{q_{F(p)}}>2^s$, the Separation Theorem (see Lemma \ref{lemhoffmann} above) implies that $p_1$ remains anisotropic over $L$. In particular, $d\big((p_1)_{L}\big) = 2^s$. By Lemma \ref{RefinedDlem}, we therefore have that
$$ d(\varphi_{L}) \leq \mathrm{dim}(\varphi) - 2^s - 2^s = \mathrm{dim}(\varphi)-2^{s+1}.$$
Rearranging, we get
\begin{eqnarray*} \witti{0}{\varphi_{L}} &=& \frac{1}{2}\left(\mathrm{dim}(\varphi) - d(\varphi_{L})\right)\\
 &\geq & \frac{1}{2}(2^{s+1})\\
 &=& 2^s. \end{eqnarray*}
Now, since $\mathrm{dim}(q) > 2^N + k$, we have 
$$ \mathrm{dim}(\tau) = \witti{0}{q_{F(p)}} = \frac{1}{2} \left(\mathrm{dim}(q) - k\right) > \frac{1}{2}(2^N) = 2^{N-1}. $$
By Lemma \ref{lemhoffmann}, it follows that 
$$ \witti{0}{\varphi_{L}} \leq \mathrm{max}(0, \mathrm{dim}(\varphi) - 2^{N-1}). $$
Since $\witti{0}{\varphi_{L}} \geq 2^s$, this implies that $\mathrm{dim}(\varphi) \geq 2^{N-1} + 2^s$, as desired. \end{proof}

\begin{corollary} The quasilinear case of Conjecture \ref{mainconj} is true when $\mathrm{dim}(q) \leq 2^{s+2} + 2^{s+1}$. 
\begin{proof} In this range, we can write $\mathrm{dim}(q) =  2^N + a$ for some non-negative integer $N$ and some integer $a \leq 2^{s+1}$.\end{proof} \end{corollary}

\noindent {\bf Acknowledgements.} The author gratefully acknowledges the support of a PIMS postdoctoral fellowship and NSERC discovery grant during the period in which this research was carried out.

\bibliographystyle{alphaurl}

\end{document}